\documentclass[a4paper,oneside,11pt]{amsproc} 

\usepackage[ansinew]{inputenc} 
\usepackage[T1]{fontenc}
\usepackage[english]{babel} 

\usepackage{enumerate} 
\usepackage{booktabs}

\usepackage{amsmath} 
\usepackage{amsfonts} 
\usepackage{amssymb}
\usepackage{amsthm}
\usepackage{latexsym}
\usepackage{bbm}
\usepackage{setspace}
\usepackage[top=3cm , bottom=3cm , left=3cm , right=3cm]{geometry}
\usepackage[pdftex]{graphicx}
  \DeclareGraphicsExtensions{.pdf,.png,.jpg,.jpeg,.mps}
  \usepackage{pgf}
  \usepackage{tikz}
\usepackage{floatflt}
\usepackage{mathrsfs}

\makeatletter

\@namedef{subjclassname@2010}{%

  \textup{2010} Mathematics Subject Classification}

\makeatother

\theoremstyle{plain}
\newtheorem{lemma}{Lemma}
\newtheorem*{theorem}{Theorem}
\newtheorem{prop}[lemma]{Proposition}
\newtheorem*{prop*}{Proposition}
\newtheorem*{lemma*}{Lemma}

\newtheorem*{theorem*}{Theorem}

\theoremstyle{definition} 

\newtheorem*{def*}{Definition}

\theoremstyle{remark}
\newtheorem*{remark}{Remarks}
\newtheorem*{acks}{Acknowledgements}

\title[An $L^1$-estimate for spectral multipliers]{An $L^1$-estimate for certain spectral multipliers associated with the Ornstein--Uhlenbeck operator}
\author{Mikko Kemppainen}
\address{Department of Mathematics and Statistics, University of Helsinki,
Gustaf H\"allstr\"omin katu 2b, FI-00014 Helsinki, Finland}
\email{mikko.k.kemppainen@helsinki.fi}

\begin{document}

\begin{abstract}
  We study a class of spectral multipliers $\phi(L)$ for the
  Ornstein--Uhlenbeck operator $L$ arising from the Gaussian measure
  on $\mathbb{R}^n$ and find a sufficient condition for integrability
  of $\phi(L)f$ in terms of the admissible conical square function
  and a maximal function.
\end{abstract}


\maketitle

%\tableofcontents

\section{Introduction}

On the Euclidean space $\mathbb{R}^n$, the \emph{Ornstein--Uhlenbeck operator}
\begin{equation*}
  L = -\frac{1}{2} \Delta + x\cdot\nabla
\end{equation*}
is associated with the Gaussian measure
\begin{equation*}
  d\gamma (x) = \pi^{-n/2} e^{-|x|^2} \, dx
\end{equation*}
by the Dirichlet form
\begin{equation*}
  \int_{\mathbb{R}^n} (Lf) g \, d\gamma = \frac{1}{2} \int_{\mathbb{R}^n}
  \nabla f \cdot \nabla g \, d\gamma , \quad f,g\in C^\infty_0(\mathbb{R}^n) .
\end{equation*}
It has discrete spectrum $\sigma (L) = \{ 0,1,2,\ldots \}$ on 
$L^2(\gamma)$, and an orthonormal
basis of eigenfunctions is given by \emph{Hermite polynomials}
$h_\beta$, $\beta\in\mathbb{N}^n$, so that $Lh_\beta = |\beta| h_\beta$.
Moreover, it generates a (positive) 
diffusion semigroup on $L^2(\gamma)$ which can be 
expressed as
\begin{equation*}
  e^{-tL}f(x) = \int_{\mathbb{R}^n} M_t(x,y) f(y) \, d\gamma (y) , \quad t>0 ,
\end{equation*}
by means of the \emph{Mehler kernel} (see \cite{S})
\begin{equation*}
  M_t(x,y) = \frac{1}{(1-e^{-2t})^{n/2}}
  \exp \Big( -\frac{e^{-t}}{1-e^{-2t}} |x-y|^2 \Big)
  \exp \Big( \frac{e^{-t}}{1+e^{-t}} (|x|^2 + |y|^2) \Big) .
\end{equation*}
The semigroup is contractive on $L^p(\gamma)$ for each 
$1\leq p \leq \infty$, and acts conservatively
so that $e^{-tL}1 = 1$.
Therefore, E. M. Stein's theory \cite{ST} is applicable in studying the
boundedness of spectral multipliers $\phi (L)$ defined as
$\phi(L)h_\beta = \phi (|\beta|) h_\beta$ for $\beta\in\mathbb{N}^n$.
More precisely, \cite[Corollary 3, p. 121]{ST} 
guarantees $L^p(\gamma)$-boundedness with
$1 < p < \infty$, for any spectral multiplier of
`Laplace transform type', i.e. of the form
\begin{equation*}
  \phi (\lambda ) = \lambda \int_0^\infty \Phi (t) e^{-t\lambda} \, dt ,
  \quad \lambda \geq 0,
\end{equation*}
where $\Phi : (0,\infty) \to \mathbb{C}$ is bounded. In particular,
the imaginary powers
$L^{i\tau}$, $\tau\in\mathbb{R}$, arising from 
$\Phi(t) = t^{-i\tau} / \Gamma(1-i\tau)$ 
are bounded on $L^p(\gamma)$.
The general theory was improved by
M. Cowling \cite{C} who showed that,
for a given $p\in (1,\infty)$,
$\phi(L)$ is bounded on $L^p(\gamma)$ as soon as $\phi$ 
extends analytically to a sector of angle greater than $\pi |1/p - 1/2|$.
(See also the more recent development \cite{CD}.)

The $L^1$-theory in the Gaussian setting is quite problematic.
Although finite linear combinations of Hermite polynomials are
dense in $L^1(\gamma)$, the spectral projections onto their eigenspaces
are not $L^1(\gamma)$-bounded. Moreover, $tLe^{-tL}$ is bounded (uniformly)
on $L^p(\gamma)$ only when $1 < p < \infty$ (see \cite[Chapter 5]{J}).

A Gaussian weak $(1,1)$-type estimate for spectral multipliers of Laplace 
transform type was established by 
J. Garc\'{\i}a-Cuerva et al. \cite{GMST}. Moreover, 
in \cite{GMMST} they showed that requiring analyticity
of $\phi$ in a sector of angle smaller than $\arcsin |1-2/p|$
will not alone suffice for boundedness of $\phi(L)$ on $L^p(\gamma)$.
Observing that $\arcsin |1-2/p| \to \pi /2$ as $p\to 1$ is in line with
the fact that the spectrum of $L$ on $L^1(\gamma)$ is the (closed)
right half-plane. Furthermore, $L^1(\gamma)$-boundedness of dilation invariant spectral multipliers
for $L$ was characterised in \cite[Theorem 3.5 (ii)]{HMM}.

The main obstruction in developing a metric theory of Hardy spaces in
the Gaussian setting arises from the fact that the 
rapidly decaying measure $\gamma$ is non-doubling, that is, for every $t>0$
\begin{equation*}
  \frac{\gamma (B(x,2t))}{\gamma (B(x,t))} \longrightarrow \infty , \quad
  \textup{as} \quad |x|\to\infty .
\end{equation*}

G. Mauceri and S. Meda overcame this problem in \cite{MM} 
and developed an atomic theory for a Gaussian Hardy space which relies
of the fact that the Gaussian measure
behaves well locally with respect to the \emph{admissibility function}
\begin{equation*}
  m(x) = \min (1 , |x|^{-1}) , \quad x\in\mathbb{R}^n .
\end{equation*}
Indeed, $\gamma$ is doubling on families of `admissible' Euclidean balls
\begin{equation*}
  \mathscr{B}_\alpha = \{ B(x,t) : 0 < t \leq \alpha m(x) \} , \quad \alpha \geq 1,
\end{equation*}
in the sense that for all $\lambda \geq 2$ we have
\begin{equation}
\label{doubling}
  \gamma (\lambda B) \leq e^{4\lambda^2\alpha^2} \gamma (B) , \quad B\in\mathscr{B}_\alpha .
\end{equation}

Other natural objects that are suitable for defining Hardy spaces, namely
maximal functions and square functions, 
were studied in the Gaussian setting by 
J. Maas, J. van Neerven, and P. Portal. 
In \cite{MNP11} they considered 
(a version\footnote{They have $t\nabla e^{-t^2L}$ instead of
$t^2L e^{-t^2L}$.} of)
the \emph{admissible conical square function}
\begin{equation*}
  Sf(x) = \Big( \int_0^{2m(x)} \frac{1}{\gamma (B(x,t))} \int_{B(x,t)} |t^2L e^{-t^2L}f(y)|^2 
  \, d\gamma (y) \, \frac{dt}{t} \Big)^{1/2} , \quad x\in\mathbb{R}^n,
\end{equation*}
and showed that it is controlled by a non-tangential semigroup maximal
function. The converse inequality was presented in
\cite{P} along with a proof that the Riesz transform
satisfies $\| \nabla L^{-1/2} f \|_1 \lesssim \| Sf \|_1 + \| f \|_1$.
The benefit of conical objects (as opposed to vertical ones) is
the applicability of tent space theory, which in the Gaussian setting
was initiated in \cite{MNP12} and further developed by
A. Amenta and the author in \cite{AK}.

The aim of this paper is to examine the decomposition method
presented in \cite{P} 
and to see what kind of $L^1$-estimates
one can obtain for spectral multipliers $\phi(L)f$ 
in terms of the admissible conical square function $Sf$ and other 
relevant objects.
The hope is that these considerations will help in developing a
fully satisfactory theory of Gaussian Hardy spaces.

\begin{theorem}
  Let
  \begin{equation*}
    \phi(\lambda) = \int_0^\infty \Phi(t) (t\lambda)^2 e^{-t\lambda} \,
    \frac{dt}{t} , \quad \lambda\geq 0,
  \end{equation*}
  where $\Phi : (0,\infty) \to \mathbb{C}$ is twice continuously differentiable 
  and satisfies
  \begin{equation*}
    \sup_{0<t<\infty} (|\Phi(t)| + t|\Phi'(t)| + t^2|\Phi''(t)|) 
    + \int_1^\infty (|\Phi'(t)| + t|\Phi''(t)|) \, dt < \infty .
  \end{equation*}
  Then, for all $f\in L^1(\gamma)$, we have
  \begin{equation*}
    \| \phi(L)f \|_1 \lesssim \| Sf \|_1 + \| f \|_1
    + \| (1+\log_+ |\cdot |) \, Mf \|_1 ,
  \end{equation*}
  where
  \begin{equation*}
    Mf(x) = \sup_{\varepsilon m(x)^2 < t \leq 1} |e^{-tL}f(x)|, \quad x\in\mathbb{R}^n ,
  \end{equation*}
  and $\varepsilon > 0$ does not depend on $f$.
\end{theorem}

\begin{remark}
Several remarks are in order:
\begin{enumerate}
  \item
  The term $\| (1+\log_+ |\cdot |) \, Mf \|_1$ is highly undesirable
  for two reasons. Firstly, the maximal operator $M$ is of
  a non-admissible kind in the sense that it is not restricted to times $t\lesssim m(\cdot)$.
  Secondly, the weight factor $(1+\log_+ |\cdot |)$, which arises
  from the admissibility function $m$, seems problematic.
  However, it is difficult to see how the appearance of the term
  could be avoided.
  Notice, nevertheless, that $\| (1+\log_+ |\cdot |) \, Mf \|_1$
  is finite at least if $f\in L^p(\gamma)$ with $1 < p < \infty$.

  \item
  The operators in the theorem above are special kind 
  of Laplace type multipliers;
  \begin{equation*}
    \phi(\lambda) = \int_0^\infty \Phi(t) (t\lambda)^2 e^{-t\lambda} \,
    \frac{dt}{t}
    = \lambda \int_0^\infty (\Phi(t) + t\Phi'(t)) e^{-t\lambda} \, dt ,
    \quad \lambda\geq 0,
  \end{equation*}
  and therefore bounded on $L^p(\gamma)$ when $1 < p < \infty$.
  Note that if, in addition, we had
  \begin{equation*}
    \int_0^1 (|\Phi'(t)| + t|\Phi''(t)|) \, dt < \infty ,
  \end{equation*}
  then $\phi(L)$ would be bounded even on $L^1(\gamma)$. Indeed,
  using integration by parts we have
  \begin{equation*}
    \phi(L)f = -\Phi(0)f + \int_0^\infty (2\Phi'(t) + t\Phi''(t))
    e^{-tL}f \, dt
  \end{equation*}
  so that $\| e^{-tL}f \|_1 \leq \| f \|_1$ implies
  \begin{equation*}
    \| \phi(L)f \|_1 \lesssim \Big( |\Phi(0)| + \int_0^\infty
    (|\Phi'(t)| + t|\Phi''(t)|) \, dt \Big) \| f \|_1 .
  \end{equation*}
  
  \item
  An example of a multiplier satisfying the conditions of the theorem is
  the localized imaginary power arising from
  $\Phi(t) = t^{i\tau} \chi(t)$, where $\tau\in\mathbb{R}$ and 
  $\chi$ is a smooth cutoff with,
  say, $1_{(0,1]} \leq \chi \leq 1_{(0,2]}$. Observe that for
  $0 < t \leq 1$ we have
  $|\Phi'(t)| \eqsim t^{-1}$ and $|\Phi''(t)| \eqsim t^{-2}$ so that
  \begin{equation*}
    \int_0^1 (|\Phi'(t)| + t|\Phi''(t)|) \, dt = \infty .
  \end{equation*}
\end{enumerate}
\end{remark}

\begin{acks}
  The research has been supported by the Academy of Finland via the Centre of Excellence in Analysis and Dynamics Research (project No. 271983). The author wishes to thank Alex Amenta and Jonas Teuwen for enlightening discussions.
\end{acks}

\section{Proof of the theorem}

\subsection*{Strategy}
The proof of the theorem follows the decomposition
method from \cite{P}.
Let us begin by introducing a
discretized version of the admissibility function
\begin{equation*}
  \widetilde{m}(x) =
  \begin{cases}
    1, &|x|<1, \\
    2^{-k}, &2^{k-1} \leq |x| < 2^k , \quad k\geq 1,
  \end{cases}
\end{equation*}
and write $\widetilde{\mathscr{B}}_\alpha$ for the associated family of
admissible balls. From 
$\widetilde{m} \leq m \leq 2\widetilde{m}$ it follows that
$\widetilde{\mathscr{B}}_\alpha \subset \mathscr{B}_\alpha \subset \widetilde{\mathscr{B}}_{2\alpha}$. This discretization is relevant for Proposition
\ref{pi3}.

We define the \emph{Gaussian tent space} adapted to this new admissibility function as the space
$\mathfrak{t}^1(\gamma)$ of functions $u$ on the \emph{admissible region} 
$D = \{ (y,t)\in \mathbb{R}^n \times (0,\infty) : 0 < t < \widetilde{m}(y) \}$ 
for which
\begin{equation*}
    \| u \|_{\mathfrak{t}^1(\gamma)} = \int_{\mathbb{R}^n} \Big( 
    \iint_{\Gamma (x)} |u(y,t)|^2 \, \frac{d\gamma (y) \, dt}{t\gamma (B(y,t))} 
    \Big)^{1/2}
    d\gamma (x) < \infty .
\end{equation*}
Here $\Gamma (x) = \{ (y,t)\in D : |y-x| < t \}$ is an admissible 
cone at $x\in \mathbb{R}^n$.

The main theorem of \cite{AK} guarantees that 
every $u\in\mathfrak{t}^1(\gamma)$ admits a decomposition into `atoms' $a_k$
so that
\begin{equation*}
  u = \sum_k \lambda_k a_k , \quad \textup{with} \quad 
  \sum_k |\lambda_k| \eqsim \| u \|_{\mathfrak{t}^1(\gamma)} .
\end{equation*}  
Recall that \emph{atom} is a function $a$ on $D$ associated with a ball 
$B\in\widetilde{\mathscr{B}}_5$ for which $\textup{supp}\, a \subset B\times (0,r_B)$ and
\begin{equation*}
  \Big( \int_0^{r_B} \| a(\cdot , t) \|_2^2 \, \frac{dt}{t} \Big)^{1/2}
  \leq \gamma (B)^{-1/2} .
\end{equation*}
For such a function, $\| a \|_{\mathfrak{t}^1(\gamma)} \lesssim 1$.
 
Let then $\phi$ and $\Phi$ be as in Theorem and let $f$ be a polynomial.
For any $\delta, \delta' > 0$ and $\kappa \geq 1$ we can decompose $\phi(L)f$ into three parts as follows:
\begin{align*}
  \phi(L)f &= c_{\delta,\delta'} \int_0^\infty \Phi((\delta'+\delta)t^2) (t^2L)^2 
               e^{-(\delta'+\delta)t^2L} f \, \frac{dt}{t} \\
    &= c_{\delta,\delta'} \Big( \int_0^{\widetilde{m}(\cdot)/\kappa} \widetilde{\Phi}(t^2) t^2L
               e^{-\delta' t^2L} u(\cdot , t) \, \frac{dt}{t} \\
    &\quad\quad + \int_0^{\widetilde{m}(\cdot)/\kappa} \widetilde{\Phi}(t^2) t^2L
               e^{-\delta' t^2L} (1_{D^c}(\cdot , t) t^2L e^{-\delta t^2L}f) \, 
               \frac{dt}{t} \\
    &\quad\quad + \int_{\widetilde{m}(\cdot)/\kappa}^\infty \widetilde{\Phi}(t^2) (t^2L)^2 
               e^{-(\delta'+\delta)t^2L}f \, \frac{dt}{t} \Big) \\
    &=: c_{\delta,\delta'} ( \pi_1 u + \pi_2 f + \pi_3 f ),
\end{align*}
where $u(\cdot , t) = 1_D(\cdot , t) t^2L e^{-\delta t^2L}f$
and $\widetilde{\Phi} (t) = \Phi ((\delta'+\delta)t)$.

Now
\begin{equation*}
  \| \phi(L)f \|_1 \leq |c_{\delta,\delta'} | (\| \pi_1 u \|_1 + \| \pi_2 f \|_1
  + \| \pi_3 f \|_1) ,
\end{equation*}
and the proof consists of estimating these three terms separately
for sufficiently small $\delta > \delta' > 0$ and large enough 
$\kappa \geq 1$.

\subsection*{Analysis of the three parts}

Proposition \ref{pi1} deals with
\begin{equation*}
  \pi_1 u = \int_0^{\widetilde{m}(\cdot)/\kappa} \widetilde{\Phi}(t^2) t^2L
               e^{-\delta' t^2L} u(\cdot , t) \, \frac{dt}{t}
\end{equation*}
and relies on the following $L^2$-$L^2$ -off diagonal estimate
(cf. \cite[Proposition 4.2]{P} and \cite{NP}).

\begin{lemma}
\label{od}
  There exists a constant $c_{od} > 0$ such that for $j=0,1$ we have
  \begin{equation*}
    \| 1_{E'} (tL)^j e^{-tL} 1_E \|_{2\to 2}
    \lesssim \exp \Big( -\frac{d(E,E')^2}{c_{od}t} \Big) , \quad t>0,
  \end{equation*}
  whenever $E,E'\subset\mathbb{R}^n$.
\end{lemma}

\begin{prop}
\label{pi1}
  Let $\kappa \geq 1$ and $0 < \delta \leq 1$. For sufficiently small
  $\delta' > 0$
  we have $\| \pi_1 u \|_1 \lesssim \| u \|_{\mathfrak{t}^1(\gamma)}$.
  Moreover, the function $u(\cdot , t) = 1_D (\cdot , t) t^2L e^{-\delta t^2L}f$ satisfies $\| u \|_{\mathfrak{t}^1(\gamma)} \lesssim \| Sf \|_1$.
\end{prop}
\begin{proof}
  By the atomic decomposition, 
  it suffices to show that $\| \pi_1 a \|_1 \lesssim 1$ for any 
  atom $a$ associated with a ball $B\in\widetilde{\mathscr{B}}_5$.
  Let us 
  consider the annuli $C_k(B) = 2^{k+1}B \setminus 2^k B$ for
  $k\geq 1$, and $C_0(B) = 2B$.
  By H\"older's inequality we have
  \begin{equation}
  \label{holder}
  \begin{split}
    \| \pi_1 a \|_1 &=
    \Big\| \int_0^{\widetilde{m}(\cdot)/\kappa} \widetilde{\Phi}(t^2) t^2L e^{-\delta' t^2L} a(\cdot , t) \, \frac{dt}{t} \Big\|_1 \\
    &\leq \sum_{k=0}^\infty \gamma (2^{k+1}B)^{1/2}
    \Big\| 1_{C_k(B)} \int_0^{r_B\wedge 2^{-k-1}/\kappa} \widetilde{\Phi} (t^2) t^2L e^{-\delta' t^2L} a(\cdot , t) \, \frac{dt}{t} \Big\|_2 .
    \end{split}
  \end{equation}
  
  We estimate the norms on the right hand side of \eqref{holder}
  by pairing with a $g\in L^2(\gamma)$  and relying on the assumption
  that $\Phi$ is bounded:
  \begin{equation*}
    \begin{split}
    &\Big| \int_{\mathbb{R}^n} \int_0^{r_B\wedge 2^{-k-1}/\kappa} \widetilde{\Phi} (t^2) t^2Le^{-\delta' t^2L} a(\cdot , t) \, \frac{dt}{t} \, g \, d\gamma \Big| \\
    &= \Big| \int_0^{r_B\wedge 2^{-k-1}/\kappa} \int_B a(\cdot , t) \widetilde{\Phi} (t^2) t^2Le^{-\delta' t^2L} g \, d\gamma \, \frac{dt}{t} \Big| \\
    &\lesssim \Big( \int_0^{r_B} \| a(\cdot , t) \|_2^2 \, \frac{dt}{t} \Big)^{1/2} \Big( \int_0^{r_B} \| 1_B t^2L e^{-\delta' t^2L} g \|_2^2 \, \frac{dt}{t}
    \Big)^{1/2} .
    \end{split}
  \end{equation*}
  
  Now, for $g$ supported in $C_0(B) = 2B$ we have
  \begin{equation*}
  \begin{split}
    \Big( \int_0^{r_B} \| 1_B t^2L e^{-\delta' t^2L}g \|_2^2 \, \frac{dt}{t} \Big)^{1/2}
    &= \Big( \sum_{\beta\in\mathbb{N}^n} |\langle g , h_\beta \rangle |^2
    \int_0^{r_B} (t^2|\beta|)^2 e^{-2\delta' t^2|\beta|} \, \frac{dt}{t} \Big)^{1/2} \\
    &\lesssim \| g \|_2 .
  \end{split}
  \end{equation*}
  
  When $k\geq 1$ we have $d(C_k(B),B) \geq (2^k - 1)r_B \geq 2^{k-1} r_B$
  and so, by Lemma \ref{od}, it follows that for $0 < t \leq r_B$,
  \begin{equation*}
    \| 1_B t^2L e^{-\delta' t^2L} 1_{C_k(B)} \|_{2\to 2} 
    \lesssim \exp \Big( -\frac{4^{k-1}r_B^2}{c_{od}\delta' t^2} \Big) 
    \lesssim \exp \Big( -\frac{4^{k-2}}{c_{od}\delta'} \Big) \Big(\frac{t}{r_B}\Big)^{1/2} .
  \end{equation*}
  Hence, for $g$ supported in $C_k(B)$, $k\geq 1$, we have
  \begin{equation*}
    \Big( \int_0^{r_B} \| 1_B t^2L e^{-\delta' t^2L} g \|_2^2 \, \frac{dt}{t} \Big)^{1/2}
    \lesssim \exp \Big( -\frac{4^{k-2}}{c_{od}\delta'} \Big) \| g \|_2 .
  \end{equation*}
  
  We have therefore shown that, for $k\geq 0$,
  \begin{equation*}
    \Big\| 1_{C_k(B)} \int_0^{r_B} \widetilde{\Phi} (t^2) t^2L e^{-\delta' t^2L} a(\cdot , t) \, \frac{dt}{t} \Big\|_2
    \lesssim \exp \Big( -\frac{4^{k-2}}{c_{od}\delta'} \Big) \gamma (B)^{-1/2} .
  \end{equation*}  
  According to the doubling inequality \eqref{doubling}, we have
  $\gamma (2^{k+1}B)^{1/2} \lesssim e^{2\cdot 4^{k+1} \cdot 25} \gamma (B)^{1/2}$ and therefore
  \begin{equation*}
    \| \pi_1 a \|_1 \lesssim \sum_{k=0}^\infty \gamma (2^{k+1}B)^{1/2}
    \exp \Big( -\frac{4^{k-2}}{c_{od}\delta'} \Big) \gamma (B)^{-1/2}
    \lesssim \sum_{k=0}^\infty \exp \Big( 50\cdot 4^{k+1} - \frac{4^{k-2}}{c_{od}\delta'} \Big) \lesssim 1 
  \end{equation*}
  as soon as $\delta' < 1/(3200c_{od})$. This proves the first claim.
  
For the second claim, let 
$u(\cdot , t) = 1_D (\cdot , t) t^2Le^{-\delta t^2L} f$. We perform a change of variable, 
$\delta t^2 = s^2$, i.e. $t = s / \sqrt{\delta}$ so that
    \begin{align*}
      (y,t)\in\Gamma (x) &\Leftrightarrow |y-x| < t < \widetilde{m}(y) \\
      &\Leftrightarrow |y-x| < s / \sqrt{\delta} < \widetilde{m}(y) \\
      &\Leftrightarrow (y,s) \in \Gamma'_{1/\sqrt{\delta}}(x) 
      := \{ (y,s)\in D' : |y-x| < s/\sqrt{\delta} \} ,
    \end{align*}
    where $D' := \{ (y,s) \in \mathbb{R}^n \times (0,\infty ) : s < \sqrt{\delta} \widetilde{m}(y) \}$.
    Now, change of aperture in the Gaussian tent space on $D'$ (see \cite[Corollary 3.5]{AK}) 
    guarantees that
    \begin{align*}
      \| u \|_{\mathfrak{t}^1(\gamma)} &= \int_{\mathbb{R}^n} \Big( \iint_{\Gamma (x)} |t^2L e^{-\delta t^2 L}f(y)|^2 
      \,\frac{d\gamma (y)\, dt}{t\gamma (B(y,t))} \Big)^{1/2} d\gamma (x) \\
      &= \int_{\mathbb{R}^n} \Big( \iint_{\Gamma'_{1/\sqrt{\delta}} (x)} 
      |\delta^{-1} s^2 L e^{-s^2L}f(y)|^2 
      \,\frac{d\gamma (y)\, ds}{s \gamma (B(y,s/\sqrt{\delta}))} \Big)^{1/2} d\gamma (x) \\
      &\lesssim \int_{\mathbb{R}^n} \Big( \iint_{\Gamma'(x)} 
      |s^2L e^{-s^2L}f(y)|^2 
      \,\frac{d\gamma (y)\, ds}{s \gamma (B(y,s))} \Big)^{1/2} d\gamma (x).
    \end{align*}
    We then observe (see \cite[Lemma 2.3]{MNP12}) 
    that for any $x,y\in\mathbb{R}^n$, $|y-x| < s < m(y)$ 
    implies $s < 2m(x)$, and therefore
    \begin{equation*}
      \Gamma'(x) \subset \Gamma (x) \subset \bigcup_{0<s<2m(x)} B(x,s) \times \{ s \} .
    \end{equation*}    
    Moreover, $\gamma (B(y,s)) \eqsim \gamma (B(x,s))$ 
    when $|y-x| < s < \delta \widetilde{m}(y)$, and hence
    \begin{equation*}
      \iint_{\Gamma'(x)} 
      |s^2L e^{-s^2L}f(y)|^2 
      \,\frac{d\gamma (y)\, ds}{s \gamma (B(y,s))}
      \lesssim \int_0^{2m(x)} \frac{1}{\gamma (B(x,s))}\int_{B(x,s)} |s^2L e^{-s^2L}f(y)|^2 
  \, d\gamma (y) \, \frac{ds}{s}
    \end{equation*}
    for every $x\in\mathbb{R}^n$, which shows that 
    $\| u \|_{\mathfrak{t}^1(\gamma)} \lesssim \| Sf \|_1$ as required.
\end{proof}

For $\pi_2$ and $\pi_3$ (more precisely, for Proposition \ref{pi2} and
Lemma \ref{admissibleL1}) we need the following two lemmas concerning pointwise kernel
estimates.

\begin{lemma}
\label{kernelest}
  Let $j=0,1$. For all $x,y\in\mathbb{R}^n$ we have the pointwise kernel estimate
  \begin{equation*}
    |t^j \partial_t^j M_t(x,y)|
    \lesssim t^{-n/2} \exp \Big( -\frac{|x-y|^2}{8t} \Big)
    \exp \Big( \frac{|x|^2 + |y|^2}{2} \Big) , \quad 0 < t \leq 1 .
  \end{equation*}
  As a consequence, for all $0 < t \leq 1$ we have
  \begin{equation*}
    \| 1_{E'} (tL)^j e^{-tL} 1_E \|_{1\to\infty}
    \lesssim t^{-n/2} \exp \Big( -\frac{d(E,E')^2}{8t} \Big)
    \sup_{\substack{x\in E \\ y\in E'}} \exp \Big( \frac{|x|^2 + |y|^2}{2} \Big) ,
  \end{equation*}
  whenever $E,E'\subset\mathbb{R}^n$.
\end{lemma}
\begin{proof}
  For $0 < t \leq 1$ we have the elementary estimates
  \begin{equation*}
    \frac{1}{1-e^{-2t}} \eqsim \frac{1}{t}, \quad
    \frac{1}{4t} \leq \frac{e^{-t}}{1-e^{-2t}} \leq \frac{1}{2t}, \quad
    \frac{1}{8} \leq \frac{e^{-t}}{1 + e^{-t}} \leq \frac{1}{2}
  \end{equation*}  
  and the case $j=0$ follows immediately:
  \begin{equation*}
  \begin{split}
    M_t(x,y) &= \frac{1}{(1-e^{-2t})^{n/2}}
  \exp \Big( -\frac{e^{-t}}{1-e^{-2t}} |x-y|^2 \Big)
  \exp \Big( \frac{e^{-t}}{1+e^{-t}} (|x|^2 + |y|^2) \Big) \\
  &\lesssim t^{-n/2}
  \exp \Big( -\frac{|x-y|^2}{4t} \Big)
  \exp \Big( \frac{|x|^2 + |y|^2}{2} \Big) .
  \end{split}
  \end{equation*}
  
  For $j=1$ we calculate:
  \begin{equation*}
    \partial_t M_t(x,y) = \Big( -\frac{ne^{-2t}}{1-e^{-2t}}
    + |x-y|^2 \frac{e^{-t}(1+e^{-2t})}{(1-e^{-2t})^2}
    - (|x|^2 + |y|^2) \frac{e^{-t}}{(1+e^{-t})^2}\Big) M_t(x,y) .
  \end{equation*}
  Using the previous case $j=0$ we then see that
  \begin{equation*}
  \begin{split}
    |t\partial_t M_t(x,y)| &\lesssim \Big( 1 + \frac{|x-y|^2}{t} + |x|^2 + |y|^2 \Big)
    M_t(x,y) \\
    &\lesssim t^{-n/2} \exp \Big( -\frac{|x-y|^2}{8t} \Big)
  \exp \Big( \frac{|x|^2 + |y|^2}{2} \Big) .
  \end{split}
  \end{equation*}
  
  The consequence is also immediate: for any $x\in E'$ we have
  \begin{equation*}
  \begin{split}
    |(tL)^je^{-tL}f(x)| 
    &\lesssim t^{-n/2}
    \int_E \exp \Big( -\frac{|x-y|^2}{8t} \Big)
    \exp \Big( \frac{|x|^2 + |y|^2}{2}\Big) |f(y)| \, d\gamma (y) \\
    &\lesssim t^{-n/2} \exp \Big( -\frac{d(E,E')^2}{8t} \Big)
    \sup_{y\in E} \exp \Big( \frac{|x|^2 + |y|^2}{2}\Big)
    \int_E |f(y)| \, d\gamma (y) .
  \end{split}
  \end{equation*}
\end{proof}

\begin{lemma}
\label{timedilation}
  For $\alpha$ large enough there exists a constant $c>0$ such that
  for all $x,y\in\mathbb{R}^n$ and all $0 < t \leq 1$ we have
  \begin{equation*}
    M_{t/\alpha}(x,y) \lesssim \exp \Big( -\frac{|x-y|^2}{ct} \Big)
    \exp \Big( \alpha t \min (|x|^2,|y|^2) \Big)
    M_t(x,y) ,
  \end{equation*}
  and, consequently,
  \begin{equation*}
    |t\partial_t M_{t/\alpha}(x,y)|
    \lesssim
    \exp \Big( \alpha t \min (|x|^2,|y|^2) \Big)
    M_t(x,y) .
  \end{equation*}
\end{lemma}
\begin{proof}
  An alternative way to express the Mehler kernel is 
  (see \cite{S})
  \begin{equation*}
    M_t(x,y) = \frac{1}{(1-e^{-2t})^{n/2}}
    \exp \Big( -\frac{|e^{-t}x - y|^2}{1-e^{-2t}} \Big) e^{|y|^2} .
  \end{equation*}
  By \cite[Lemma 3.4]{P} for $\alpha$ large enough
  we have for all $x,y\in\mathbb{R}^n$ and all $0 < t \leq 1$ that
  \begin{equation*}
    \exp \Big( -\frac{|e^{-t/\alpha}x - y|^2}{1-e^{-2t/\alpha}} \Big)
    \leq \exp \Big( -2\frac{|e^{-t}x - y|^2}{1-e^{-2t}} \Big)
    \exp \Big( \frac{t^2\min (|x|^2,|y|^2)}{1-e^{-2t/\alpha}} \Big) .
  \end{equation*}
  Therefore
  \begin{equation*}
    M_{t/\alpha}(x,y) \lesssim \exp \Big( -\frac{|e^{-t}x - y|^2}{1-e^{-2t}} \Big) \exp \Big( \frac{t^2\min (|x|^2,|y|^2)}{1-e^{-2t/\alpha}} \Big) M_t(x,y),
  \end{equation*}
  where, by symmetry, the first exponential factor can be replaced by
  \begin{equation*}
    \exp \Big( -\frac{\max(|e^{-t}x - y|^2, |x-e^{-t}y|^2)}{1-e^{-2t}} \Big) .
  \end{equation*}
  The first claim now follows because for all
  $x,y\in\mathbb{R}^n$ and all $0 < t \leq 1$ we have
  \begin{equation*}
    |x-y|^2 \lesssim \max(|e^{-t}x - y|^2, |x-e^{-t}y|^2) .
  \end{equation*}
  In order to see this,
  let us assume, with no loss of generality, that $|x|\leq |y|$,
  and show that $|x-y|^2 \lesssim |e^{-t}x - y|^2$. Then
  \begin{equation*}
  \begin{split}
    |x-y|^2 &\leq e ( e^{-t}|x|^2 - 2e^{-t} x\cdot y + e^{-t}|y|^2) \\
    &=e(e^{-t}|x|^2 - (1-e^{-t})|y|^2 - 2e^{-t}x\cdot y + |y|^2) ,
  \end{split}  
  \end{equation*}
  where
  \begin{equation*}
    e^{-t}|x|^2 - (1-e^{-t})|y|^2 \leq e^{-2t}|x|^2 ,
  \end{equation*}
  because $|x| \leq |y|$. Indeed,
  \begin{equation*}
    e^{-t}|x|^2 - (1-e^{-t})|x|^2 - e^{-2t}|x|^2
    =(2e^{-t} - 1 - e^{-2t})|x|^2 ,
  \end{equation*}
  where $2e^{-t} - 1 - e^{-2t} \leq 0$ for all $t>0$.
  
  The second claim now follows from the first one since
  \begin{equation*}
  \begin{split}
    |t\partial_t M_{t/\alpha}(x,y)| &\lesssim \Big( 1 + \frac{|x-y|^2}{t} + |x|^2 + |y|^2 \Big)
    M_{t/\alpha}(x,y) \\
    &\lesssim \exp \Big( \alpha t \min (|x|^2,|y|^2) \Big)
    M_t(x,y) .
  \end{split}
  \end{equation*}
  Here the first inequality is obtained as in the proof of Lemma \ref{kernelest} (case $j=1$).
\end{proof}

Let us then consider
\begin{equation*}
  \pi_2 f = \int_0^{\widetilde{m}(\cdot)/\kappa} \widetilde{\Phi}(t^2) t^2L
               e^{-\delta' t^2L} (1_{D^c}(\cdot , t) t^2L e^{-\delta t^2L}f) \, 
               \frac{dt}{t} .
\end{equation*}

\begin{prop}
\label{pi2}
  Let $\kappa \geq 4$.
  For sufficiently small $\delta > \delta' > 0$
  we have $\| \pi_2f \|_1 \lesssim \| f \|_1$.
\end{prop}
\begin{proof}
  We begin by observing that if
  $t\leq \widetilde{m}(x)/4$ and $t > 2^{-k-1}$ for some $k\geq 2$, 
  then $|x| < 2^{k-2}$.
  Moreover, if
  $t\geq \widetilde{m}(y)$ and $t\leq 2^{-k}$, then $|y| \geq 2^{k-1}$.
  
  We then decompose $\pi_2f$ (using boundedness of $\Phi$) as follows:
  \begin{equation}
  \label{pi2dec}
    \begin{split}
    \| \pi_2 f \|_1 &=
    \Big\| \int_0^{\widetilde{m}(\cdot)/\kappa} \widetilde{\Phi}(t^2) t^2Le^{-\delta' t^2L}
    (1_{D^c}(\cdot , t) t^2Le^{-\delta t^2L}f ) \, \frac{dt}{t} \Big\|_1 \\
    &\lesssim \sum_{k=2}^\infty \int_{2^{-k-1}}^{2^{-k}} 
    \| 1_{B(0,2^{k-2})} t^2L e^{-\delta' t^2L}
    (1_{\mathbb{R}^n\setminus B(0,2^{k-1})} t^2Le^{-\delta t^2L}f ) \|_1 \, \frac{dt}{t}\\
    &\leq \sum_{k=2}^\infty \sum_{l=1}^\infty
    \int_{2^{-k-1}}^{2^{-k}} \| 1_{B(0,2^{k-2})} t^2L e^{-\delta' t^2L}
    (1_{C_{k+l-1}} t^2Le^{-\delta t^2L}f) \|_1 \, \frac{dt}{t} ,
    \end{split}
  \end{equation}
  where $C_{k+l-1} := B(0,2^{k+l-1})\setminus B(0,2^{k+l-2})$.
  
  First, by Lemma \ref{timedilation}, we choose a $\delta > 0$
  such that for all $0 < t \leq 1$ we have
  \begin{equation*}
    |t^2L e^{-\delta t^2L}f(x)| 
    \lesssim \exp \Big( \frac{t^2|x|^2}{\delta} \Big) |e^{-tL}f(x)| ,
    \quad x\in\mathbb{R}^n .
  \end{equation*}
  Hence, for $2^{-k-1} < t \leq 2^{-k}$ we have
  \begin{equation*}
    \| 1_{C_{k+l-1}} t^2L e^{-\delta t^2L}f \|_1
    \lesssim \exp \Big( \frac{4^{-k} \cdot 4^{k+l-1}}{\delta} \Big)
    \| e^{-tL} f \|_1
    \lesssim \exp \Big( \frac{4^{l-1}}{\delta} \Big) \| f \|_1 .
  \end{equation*}
  
  Then, since the distance between $B(0,2^{k-2})$ and $C_{k+l-1}$ is
  at least $2^{k+l-3}$, we have, by Lemma \ref{kernelest},
  for $2^{-k-1} < t \leq 2^{-k}$ that
  \begin{equation*}
    \begin{split}
      \| 1_{B(0,2^{k-2})} t^2L e^{-\delta' t^2L} 1_{C_{k+l-1}} \|_{1\to 1}
      &\lesssim t^{-n} \exp \Big( -\frac{4^{k+l-3}}{8\delta' t^2} \Big)
      \exp \Big( \frac{4^{k-2} + 4^{k+l-1}}{2} \Big) \\
      &\lesssim 2^{kn} \exp \Big( -\frac{4^{2k+l-5}}{\delta'} + 4^{k+l-1} \Big) .
    \end{split}
  \end{equation*}
  
  Combining the two estimates we see that for $2^{-k-1} < t \leq 2^{-k}$
  we have
  \begin{equation*}  
    \begin{split}
      &\| 1_{B(0,2^{k-2})} t^2L e^{-\delta' t^2L} (1_{C_{k+l-1}} t^2L e^{-\delta t^2L}f) \|_1 \\
      &\lesssim 2^{kn} \exp \Big( -\frac{4^{2k+l-5}}{\delta'} + 4^{k+l-1} + \frac{4^{l-1}}{\delta} \Big) \| f \|_1 \\
      &= 2^{kn} \exp \Big( -4^{k+l+1} \Big( \frac{4^{k-6}}{\delta'} - 4^{-2} - \frac{4^{-2}}{\delta} \Big) \Big) \| f \|_1 \\
      &\lesssim \exp (-4^{k+l}) \| f \|_1 ,
    \end{split}
  \end{equation*}
  where in the last step we chose $\delta' < \delta$ small enough.
  
  The right-hand side of \eqref{pi2dec} is therefore dominated by
  \begin{equation*}
    \sum_{k=2}^\infty \sum_{l=1}^\infty \exp (-4^{k+l}) \| f \|_1 \int_{2^{-k-1}}^{2^{-k}} \frac{dt}{t}
    \lesssim \| f \|_1 .
  \end{equation*} 
\end{proof}

\begin{lemma}
\label{admissibleL1}
  For any $\alpha > 0$ we have
  \begin{equation*}
    \| (e^{-tL}f)|_{t={\widetilde{m}(\cdot)^2/\alpha}} \|_1
    \lesssim \| f \|_1 .
  \end{equation*}
  Moreover, for $\alpha$ large enough we have
  \begin{equation*}
    \| (tLe^{-tL}f)|_{t={\widetilde{m}(\cdot)^2/\alpha}} \|_1
    \lesssim \| f \|_1 .
  \end{equation*}
\end{lemma}
\begin{proof}
  Write
  $C_0 = B(0,1)$ and $C_k = B(0,2^k)\setminus B(0,2^{k-1})$ for $k\geq 1$.
  Moreover, let $C_0^* = B(0,2)$, $C_1^* = B(0,4)$, and
  $C_k^* = B(0,2^{k+1})\setminus B(0,2^{k-2})$ for $k\geq 2$.
  
  We first show that for any $\alpha > 0$,
  \begin{equation}
  \label{alphaadm}
    \| (e^{-tL}f)|_{t={\widetilde{m}(\cdot)^2}/\alpha} \|_1
    \lesssim \| f \|_1 .
  \end{equation}
  Denote $\varepsilon = 1/\alpha$ for notational convenience.
  For $x\in C_k$ we have $\widetilde{m}(x)^2 = 4^{-k}$ and hence  
  \begin{equation*}
      \| (e^{-tL}f)|_{t=\varepsilon\widetilde{m}(\cdot)^2} \|_1
      = \sum_{k=0}^\infty \| 1_{C_k} e^{-\varepsilon 4^{-k} L} f \|_1 .
  \end{equation*}
  We split $f$ into $1_{C_k^*}f$ and $1_{\mathbb{R}^n\setminus C_k^*} f$, and
  first estimate
  \begin{equation*}
    \sum_{k=0}^\infty \| 1_{C_k} e^{-\varepsilon 4^{-k}L} (1_{C_k^*}f) \|_1
    \leq \sum_{k=0}^\infty \| 1_{C_k^*} f \|_1 \lesssim \| f \|_1 .
  \end{equation*}
  Fixing an integer $N$ for which $8\varepsilon \leq 4^N$,
  we use the trivial estimate for $k=0,1,\ldots , N+3$:
  \begin{equation*}
    \| 1_{C_k} e^{-\varepsilon 4^{-k}L} (1_{\mathbb{R}^n\setminus C_k^*} f) \|_1
    \leq \| f \|_1 .
  \end{equation*}
  For $k\geq N+4$ we have the decomposition
  \begin{equation*}
    \mathbb{R}^n \setminus C_k^* = B(0,2^{k-2}) \cup \bigcup_{l=2}^\infty 
    C_{k+l} .
  \end{equation*}
  Observing that
  $d(C_k , B(0,2^{k-2})) = 2^{k-2}$ we obtain, by Lemma \ref{kernelest},
  \begin{equation*}
    \begin{split}
    \| 1_{C_k} e^{-\varepsilon 4^{-k}L} 1_{B(0,2^{k-2})} \|_{1\to 1}
    &\lesssim 2^{kn} \exp \Big( - \frac{4^{k-2}}{8\varepsilon 4^{-k}} \Big)
    \exp \Big( \frac{4^k + 4^{k-2}}{2} \Big) \\
    &\leq 2^{kn} \exp ( - 4^{2k-2-N} + 4^k ) \\
    &\lesssim \exp (-4^k) .
    \end{split}
  \end{equation*} 
  Furthermore, since $d(C_k , C_{k+l}) = 2^{k+l-2}$,
  Lemma \ref{kernelest} implies that
  \begin{equation*}
  \begin{split}
     \| 1_{C_k} e^{-\varepsilon 4^{-k}L} 1_{C_{k+l}} \|_{1\to 1}
     &\lesssim 2^{kn} \exp \Big( -\frac{4^{k+l-2}}{8\varepsilon 4^{-k}}\Big) \exp \Big( \frac{4^k + 4^{k+l}}{2} \Big) \\
     &\leq 2^{kn} \exp ( -4^{2k+l-2-N} + 4^{k+l}) \\
     &\lesssim \exp (-4^{k+l}) .
     \end{split}
  \end{equation*}
  Therefore,
  \begin{equation*}
  \begin{split}
    \sum_{k=N+4}^\infty \| 1_{C_k} e^{-\varepsilon 4^{-k}L} (1_{\mathbb{R}^n\setminus C_k^*} f) \|_1 &= \sum_{k=N+4}^\infty \Big( \| 1_{C_k} e^{-\varepsilon 4^{-k}L} (1_{B(0,2^{k-2})} f) \|_1 \\
    &\quad \quad \quad + \sum_{l=2}^\infty \| 1_{C_k} e^{-\varepsilon 4^{-k}L} (1_{C_{k+l}} f) \|_1 \Big) \\
    &\lesssim \sum_{k=N+4}^\infty \Big( \exp (-4^k) \| f \|_1 + \sum_{l=2}^\infty
    \exp (-4^{k+l})\| f \|_1 \Big) \\
    &\lesssim \| f \|_1.
    \end{split}
  \end{equation*}
  We have now proven \eqref{alphaadm} which includes the first case
  from the statement of the lemma.
  
  The second case follows by using
  Lemma \ref{timedilation}, which guarantees that 
  there exists 
  an $\alpha > 0$ such that for all $x,y\in\mathbb{R}^n$
  \begin{equation*}
    \Big| (t\partial_t M_t(x,y))|_{t=\widetilde{m}(x)^2/\alpha} \Big|
    \lesssim \exp \Big( \alpha \frac{\widetilde{m}(x)^2}{\alpha} |x|^2 \Big)  
    M_{\widetilde{m}(x)^2}(x,y) \lesssim M_{\widetilde{m}(x)^2}(x,y) .
  \end{equation*}
  Then
  \begin{equation*}
    \| (tL e^{-tL}f)|_{t=\widetilde{m}(\cdot)^2 / \alpha} \|_1
    \lesssim \| (e^{-tL}f)|_{t=\widetilde{m}(\cdot)^2} \|_1
    \lesssim \| f \|_1 .
  \end{equation*}
\end{proof}

Finally, we turn to
\begin{equation*}
  \pi_3f = \int_{\widetilde{m}(\cdot)/\kappa}^\infty 
  \widetilde{\Phi}(t^2) (t^2L)^2 
               e^{-(\delta'+\delta)t^2L}f \, \frac{dt}{t} .
\end{equation*}

\begin{prop}
\label{pi3}
  Let $0 < \delta,\delta' \leq 1/2$.
  For $\kappa$ large enough
  we have $\| \pi_3f \|_1 \lesssim \| f \|_1 + \| (1+\log_+ |\cdot |) \, Mf \|_1$,
  where $Mf(x) = \sup_{\varepsilon m(x)^2<t\leq 1} |e^{-tL}f(x)|$ and $\varepsilon > 0$ does not depend on $f$.
\end{prop}
\begin{proof}
  Integrating by parts we obtain
  \begin{equation*}
  \begin{split}
    &\int_{\widetilde{m}(\cdot)/\kappa}^\infty \widetilde{\Phi} (t^2) (t^2L)^2 
    e^{-(\delta'+\delta)t^2L}f \, \frac{dt}{t} \\
    &= c \int_{\widetilde{m}(\cdot)^2/\kappa^2}^\infty \widetilde{\Phi} (t) t\partial_t^2 
    e^{-(\delta'+\delta)tL}f \, dt \\
    &= c \Big[ \widetilde{\Phi} (t) t\partial_t e^{-(\delta'+\delta)tL}f \Big]_{t=\widetilde{m}(\cdot)^2/\kappa^2}^\infty + c' \int_{\widetilde{m}(\cdot)^2/\kappa^2}^\infty
    (\widetilde{\Phi} (t) + t\widetilde{\Phi}'(t)) \partial_t e^{-(\delta'+\delta)tL}f \, dt .
    \end{split}
  \end{equation*}
  Repeating for the last term we get
  \begin{equation*}
  \begin{split}
    &\int_{\widetilde{m}(\cdot)^2/\kappa^2}^\infty
    (\widetilde{\Phi} (t) + t\widetilde{\Phi}'(t)) \partial_t e^{-(\delta'+\delta)tL}f \, dt \\
    &= c \Big[ (\widetilde{\Phi} (t) + t\widetilde{\Phi}'(t)) e^{-(\delta'+\delta)tL}f \Big]_{t=\widetilde{m}(\cdot)^2/\kappa^2}^\infty + c' \int_{\widetilde{m}(\cdot)^2/\kappa^2}^\infty
    (2\widetilde{\Phi}'(t) + t\widetilde{\Phi}''(t)) e^{-(\delta'+\delta)tL}f \, dt .
    \end{split}
  \end{equation*}
  Now, having assumed that 
  $\sup_{0 < t < \infty} (|\Phi(t)| + t|\Phi'(t)|) < \infty$,
  we may use Lemma \ref{admissibleL1} to choose $\kappa$ large enough so 
  that
  \begin{equation*}
    \Big\| \Big[ \widetilde{\Phi} (t) t\partial_t e^{-(\delta'+\delta)tL}f \Big]_{t=\widetilde{m}(\cdot)^2/\kappa^2}^\infty \Big\|_1
    \lesssim \| (tLe^{-(\delta'+\delta)tL}f)|_{t={\widetilde{m}(\cdot)^2/\kappa^2}} \|_1
    \lesssim \| f \|_1
  \end{equation*}
  and
   \begin{equation*}
    \Big\| \Big[ (\widetilde{\Phi} (t) + t\widetilde{\Phi}'(t)) e^{-(\delta'+\delta)tL}f \Big]_{t=\widetilde{m}(\cdot)^2/\kappa^2}^\infty \Big\|_1
    \lesssim \| (e^{-(\delta'+\delta)tL}f)|_{t={\widetilde{m}(\cdot)^2/\kappa^2}} \|_1
    \lesssim \| f \|_1 .
  \end{equation*}
  Moreover,
  \begin{equation*}  
    \Big\| \int_1^\infty
    (2\widetilde{\Phi}'(t) + t\widetilde{\Phi}''(t)) e^{-(\delta' +\delta)tL}f \, dt \Big\|_1
    \lesssim \int_1^\infty (|\widetilde{\Phi}'(t)| + t|\widetilde{\Phi}''(t)|) 
    \| e^{-(\delta' +\delta)tL}f \|_1 \, dt \lesssim \| f \|_1 .
  \end{equation*}
  Finally, having assumed that 
  $\sup_{0 < t < \infty} (t|\Phi'(t)| + t^2|\Phi''(t)|) < \infty$,
  we get
  \begin{equation*}
  \begin{split}
    \Big| \int_{\widetilde{m}(\cdot)^2/\kappa^2}^1
    (2\widetilde{\Phi}'(t) + t\widetilde{\Phi}''(t)) e^{-(\delta'+\delta)tL}f \, dt \Big|
    &\lesssim 
    \int_{\widetilde{m}(\cdot)^2/\kappa^2}^1
    (|\widetilde{\Phi}'(t)| + t|\widetilde{\Phi}''(t)|) \, |e^{-(\delta'+\delta)tL}f| \, dt \\
    &\lesssim \sup_{\varepsilon m(\cdot)^2<t\leq 1} |e^{-tL}f|
     \int_{\widetilde{m}(\cdot)^2/\kappa^2}^1 \frac{dt}{t} \\
    &\lesssim (1 + \log_+ |\cdot |) \, Mf ,
  \end{split}  
  \end{equation*}
  where $\varepsilon > 0$ is chosen small enough depending on $\delta$, $\delta'$ and $\kappa$.
\end{proof}

\bibliographystyle{plain}
\def\cprime{$'$} \def\cprime{$'$}

\end{document}